\documentclass{article}
\usepackage[T1]{fontenc}
\usepackage[utf8]{inputenc}
\usepackage{amsthm}
\usepackage{amssymb}
\usepackage{amsmath}
\usepackage{mathtools}
\usepackage{enumitem}
\usepackage{changepage}
\usepackage{siunitx}
\usepackage{hyperref}
\usepackage{graphicx}
\usepackage{listings}
\newtheorem{theorem}{Theorem}
\newtheorem{corollary}{Corollary}
\newtheorem*{corollaryo}{Corollary}
\newtheorem*{lemmao}{Lemma}

\newtheorem{proposition}{Proposition}
\author{M.~Hellus
\and A.~Rechenauer
\and R.~Waldi}
\title{On the order of magnitude of certain integer sequences\\[2ex]       
{\normalfont \normalsize In memory of our revered teacher Ernst Kunz.}}
\begin{document}
\maketitle
\begin{abstract}
Let $p$ be a prime number, and let $S$ be the numerical semigroup generated by the prime numbers not less than $p$. We compare the orders of magnitude of some invariants of $S$ with each other, e.\,g., the biggest atom $u$ of $S$ with $p$ itself:

By Harald Helfgott (\cite{helfgott2014}), every odd integer $N$ greater than five can be written as the sum of three prime numbers. There is numerical evidence suggesting that the summands of $N$ always can be chosen between $\frac N6$ and $\frac N2$. This would imply that $u$ is less than $6p$.
\end{abstract}
\section{Notation}
Let $\pi(x)$ be the number of primes less than or equal to $x$ (\emph{prime-counting function}), $p=p(n)$ the prime number with $n=\pi(p)$ and $S=S(n)$ the additive monoid (\emph{numerical semigroup}) generated by the prime numbers not less than $p$. Those positive integers which do not belong to $S$ are called \emph{gaps}, their number $g=g(n)$ is the \emph{genus} of $S$. The \textit{Frobenius number} $f=f(n)$ of $S$ is the largest gap of $S$. Hence $s=s(n)\coloneqq f(n)+1-g(n)$ equals the number of elements of $S$ which are less than $f$; these elements are called \emph{sporadic} for $S$. The \emph{embedding dimension} $e=e(n)$ of $S$ ist the number of \emph{atoms} of $S$, where the latter are those elements of $S$ which cannot be written as the sum of two positive elemente of $S$. By $u=u(n)$ we denote the greatest atom of $S$.

Pairs $(v,w)=(v(n),w(n))$ of sequences of positive real numbers are called \emph{of the same order of magnitude} if $\inf(\frac vw)>0$ and $\sup(\frac vw)<\infty$. The functions $v$ and $w$ are called \emph{asymptotically equal}, $v\sim w$ for short, if $\lim_{n\to\infty}\frac{v(n)}{w(n)}=1$. Trivially, asymptotically equal sequences are of the same order of magnitude.
\section{Results}
We are going to present asymptotical statements about the pairs $(e,n), (f,p), (g,p), (s,p)$ and $(u,p)$.
\begin{proposition}
\begin{enumerate}[label=\alph*)]
\item $g\sim\frac52p$, \cite[Prop. 6]{hrw2020}.
\end{enumerate}
\vspace{.2cm}
There exists $n(0)>0$ such that for all $n\geq n(0)$:
\begin{adjustwidth}{2.5em}{0pt}
\begin{enumerate}
\item[b)] $3p-6\leq f<4p+p^{0.7}$.
\item[c)] $|\frac up-3|<\frac1{p^{0.3}}$. In particular, $u\sim 3p$.
\end{enumerate}
\xdef\tpd{\the\prevdepth}
\end{adjustwidth}
\begin{enumerate}
\item[d)] $e\sim 2n$, \cite[Thm. 1]{hrw2020}.
\end{enumerate}
\end{proposition}
\begin{corollaryo}
$f,g,s=f+1-g,u$ and $p$ are of the same order of magnitude.
\end{corollaryo}
For the minimums of the quotients $e/n, f/p, g/p$ and $u/p$ one has:
\begin{proposition}
\begin{enumerate}[label=\alph*)]
\item $\min(f/p)=\min(g/p)=\frac12$ (this is clear).
\item $\min(u/p)=\frac32$.
\item $\min(e/n)=\frac54$; $e(n)>n$ and $s(n)>n-1$ for all $n>0$.
\end{enumerate}
\end{proposition}
\section{Proofs}
For the proof of proposition 1 we need the following improvement of \cite[Le\-mma~4.1]{hrw2021}:
\begin{lemmao}
There exists $n(0)>0$ such that for all $n\geq n(0)$ one has: Every odd number $N\geq3p(n)+p(n)^{0.7}$ is the sum of three primes from $S(n+1)$, in particular one has
\[f(n+1)<3p(n)+p(n)^{0.7}\text{ if }f(n+1)\text{ is odd and}\]
\[f(n+1)<3p(n)+p(n+1)+p(n)^{0.7}\text{ if }f(n+1)\text{ is even.}\]
\end{lemmao}
\begin{proof}
By \cite[Thm.~1.1]{mmv2017}, there exists $n(0)>0$ such that for all $n\geq n(0)$ one has:

Every odd number $N\geq V(n)\coloneqq3p(n)+p(n)^{0.7}$ can be written as a sum
\[N=q(1)+q(2)+q(3)\]
of primes $q(i)$ with
\[|\frac N3-q(i)|\leq N^{0.6}\text{ for }i=1,2,3\,.\]
Since
\[F(x)\coloneqq \frac x3-x^{0.6}\]
is monotonically increasing for $x\geq5$ and
\[p(n)^{0.1}>12\]
holds for almost all $n$, by increasing $n(0)$ if necessary, we get
\begin{align*}q(i)&\geq\frac N3-N^{0.6}\\&\geq\frac{V(n)}3-V(n)^{0.6}\\&>p(n)+\frac{p(n)^{0.7}}3-(4p(n))^{0.6}\\&>p(n)\end{align*}
for $i=1,2,3$. Therefore, $q(i)\in S(n+1)$.
\end{proof}
Prop. 1\,b) now follows from \cite[Prop.~1]{hrw2020} and the above lemma.
\begin{proof}[Proof of Prop. 1\,c)]
By the above lemma, for $n\gg0$ the odd prime number $u(n)$ from $S(n)$ is smaller than $3p(n)+p(n)^{0.7}$.

Conversely, for $n\gg0$, by \cite[Thm. 1.1]{bhp2001}, the interval
\[[3p(n)-(3p(n))^{\frac{21}{40}},3p(n)]\]
contains at least one prime number $q$; the latter is clearly an atom of $S(n)$. Hence we get
\begin{align*}u(n)&\geq3p(n)-(3p(n))^{\frac{21}{40}}\\&>3p(n)-p(n)^{0.7}\end{align*}
for $n\geq n(0)$ for some suitable $n(0)$.
\end{proof}
\begin{proof}[Proof of Prop. 2\,b)]
One has
\[\frac{u(1)}{p(1)}=\frac32\,.\]
From now on, let $n>1$. Since then the prime numbers from $[p(n),3p(n)]$ are atoms of $S(n)$, one has $\pi(u(n))\geq\pi(3p(n))$.

The calculations from appendix \ref{App_A} show that, for $n$ between $2$ and $428$, one has $\frac{u(n)}{p(n)}>2$.

For $n>428$, by the proof of \cite[Prop.~5]{hrw2020} one has
\[\pi(3p(n))>2n\,;\]
furthermore, one has
\[2n>\pi(2p(n))\]
by Rosser and Schoenfeld \cite{rs1975}. In total, one has
\[\pi(u(n))>\pi(2p(n))\]
and hence also
\[u(n)>2p(n)\text{, i.\,e. }\frac{u(n)}{p(n)}>2\,.\]
\end{proof}
\begin{proof}[Proof of Prop. 2\,c)]
The calculations from appendix \ref{App_A} show that, for $0<n<7496$, one has $e(n)>n$ and $s(n)>n-1$; furthermore, for these $n$, $\frac{e(8)}8=\frac54$ is the minimal value of $\frac{e(n)}n$.

Now let $n=\pi(p)>\num[group-separator={,}]{7495}$ and let $p'$ be the predecessor of $p$, i.\,e. $p'=p(n-1)\geq p(\num[group-separator={,}]{7495})$. One has
\[q\coloneqq\num[group-separator={,}]{76129}=p(\num[group-separator={,}]{7495})\,.\]
Consider the function
\[\lambda(x)\coloneqq\frac{3(\log(x)-\frac32)}{\log(3x)-\frac12}\text{ for }x>1\,.\]
$\lambda(x)$ is monotonically increasing. Hence, as easily checked by calculation,
\begin{align*}\tag{$*$}\lambda(p')(n-1)&\geq\lambda(q)(n-1)\\&>2.46(n-1)\\&>2.4n\end{align*}
(since $n>\num[group-separator={,}]{7495}$).

As seen in the proof of \cite[Prop.~5]{hrw2020}, because of $(*)$ and since $3p'\geq67$, we also have
\begin{align*}\tag{$**$}\pi(3p')&>\pi(p')\lambda(p')\\&>2.4n\,,\end{align*}
where the latter holds because of $(*)$.

Furthermore, we have
\[f\geq 3p-6\geq3p'\]
by \cite[Prop.~1]{hrw2020}.

Consequently, the prime numbers in $[p,3p']$ are both atoms and sporadic for $S$.

Finally, from $(**)$ we get:
\begin{align*}e(n)&\geq\#\text{\,primes in }[p,3p']\\&=\pi(3p')-\pi(p)+1\\&>1.4n\\&>\frac54n\,.\end{align*}
Analogously, $s(n)>1.4n$.
\end{proof}
\section{Conjectures}
The calculations from appendix \ref{App_A} show that the following assertions hold true: For $n<7496$, we have
\[\frac{u(n)}{p(n)}\leq\frac{163}{47}\text{ with equality for }n=15\,,\]
\[u(n)-3p(n)<p(n)^{0.7}\text{ for }46<n<7496\text{ (and for }n\gg0\text{ by Prop. 1\,c));}\]
\[\frac{f(n)}{p(n)}\leq\frac{63}{19}\text{ with equality for }n=8\,,\]
\[f(n)-3p(n)\leq n\text{ for }n<7496\text{ with equality only for }n=30\,.\]
For $30<n<7496$, we even have
\[f(n)-3p(n)<p(n)^{0.7}\,.\]
These inequalities strongly indicate that the following conjectures hold true:
\[\tag{V1}\sup(\frac up)=\max(\frac up)=\frac{163}{47}\]
\[\tag{V2}\sup(\frac fp)=\max(\frac fp)=\frac{63}{19}\]

Similar considerations lead to the following conjecturees:
\[\tag{V3}\sup(\frac gp)=\frac52\]
\[\tag{V4}\inf(\frac sp)=\min(\frac sp)=\frac12\text{ and }\sup(\frac sp)=\max(\frac sp)=\frac{24}{19}\]

\paragraph{Conclusion} We cannot prove any concrete upper bounds for the functions under consideration.

The following considerations on the ternary Goldbach problem shall enable us to prove at least conjecture (V1) under certain assumptions.
\section{Remark on the ternary Goldbach problem}
By Harald Helfgott (\cite{helfgott2014}), every odd integer $N$ greater than five can be written as the sum of three prime numbers.

It is known that for large $N$ these prime numbers can be chosen very closely to $\frac N3$, cf. e.\,g. \cite[Thm.~1.1]{mmv2017}.

With regard to an upper bound for $\frac up$, the following weak, but effective version of \cite[Thm.~1.1]{mmv2017} would be helpful.

More generally, for reals $t>3$ and odd integers $K\geq7$ we define:

$H(K,t)$: Every odd integer $N\geq K$ can be written as the sum $N=q(1)+q(2)+q(3)$ of prime numbers $q(i)$ with the restriction
\[|\frac N3-q(i)|\leq\frac Nt\text{ for }i=1,2,3.\]
By \cite[Thm.~1.1]{mmv2017}, for each given $t>3$, there exists $K\geq7$ such that $H(K,t)$ holds. Calculations suggest that there is no counterexample to $H(7,6)$, in other words:

Every odd integer $N\geq7$ is the sum $N=q(1)+q(2)+q(3)$ of prime numbers $q(i)$ such that
\[\tag{$*$}\frac N6\leq q(i)\leq\frac N2\text{ for }i=1,2,3.\]
Notice, that here equality cannot happen by reasons of parity.

Compared to $H(7,6)$, one can easily check by calculations that assertion $H(7,7)$ holds true for $7\leq N<\num[group-separator={,}]{10000}$, with the exceptions $23$ and $27$. Further calculations show: The odd integers $23$ and $27$ fulfil $H(7,\frac{27}4)$ as well---and $6<\frac{27}4<7$---, whereas $27$ is an exception for $H(7,t)$, whenever $t>\frac{27}4$.

\begin{center}
  \includegraphics[height=10cm]{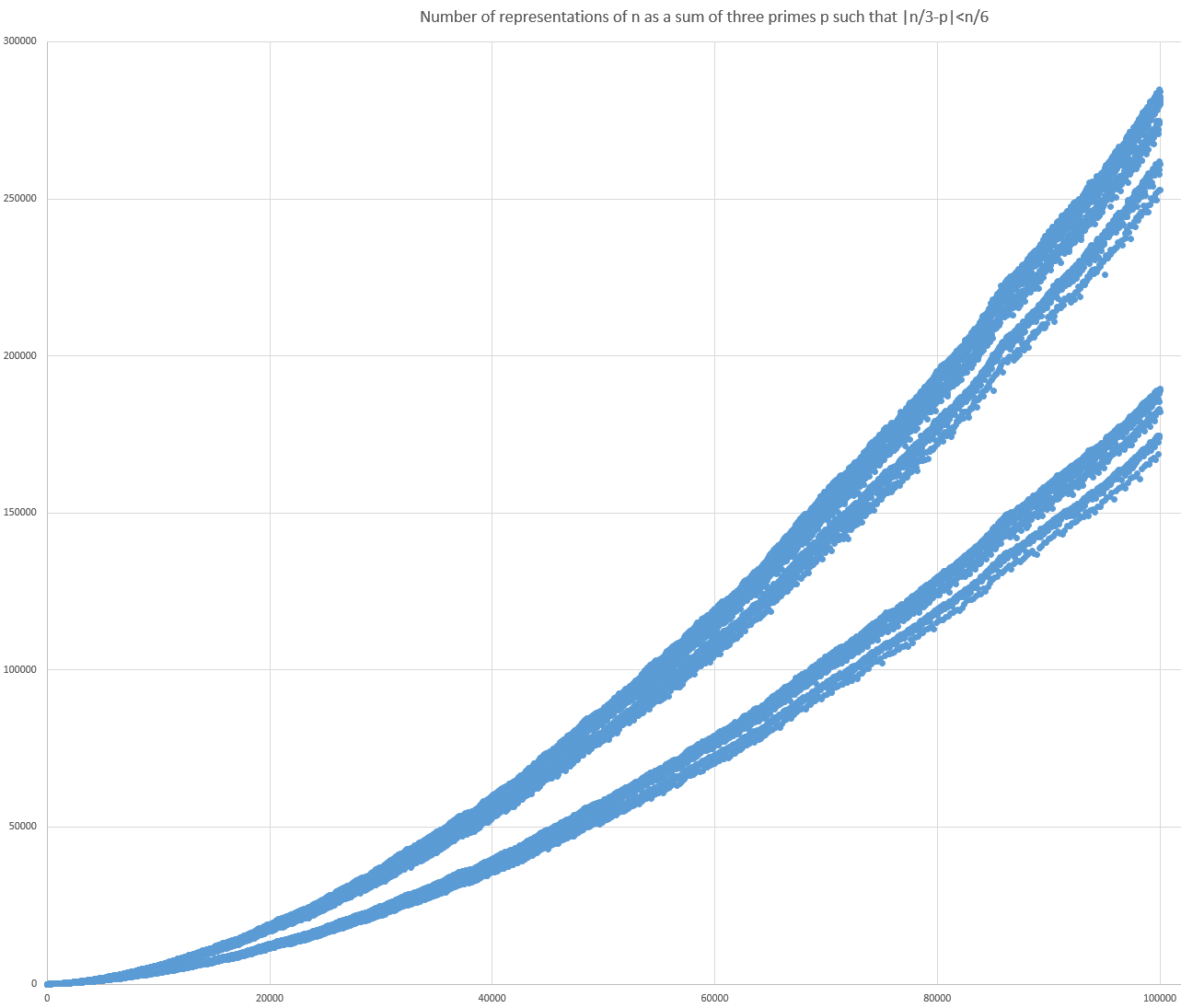}
\end{center}

This chart displays the number $r(N)$ of representations $N=q(1)+q(2)+q(3)$ according with $(*)$, as a function of $N$ between $7$ and $\num[group-separator={,}]{99999}$.

\begin{center}
  \includegraphics[height=10cm]{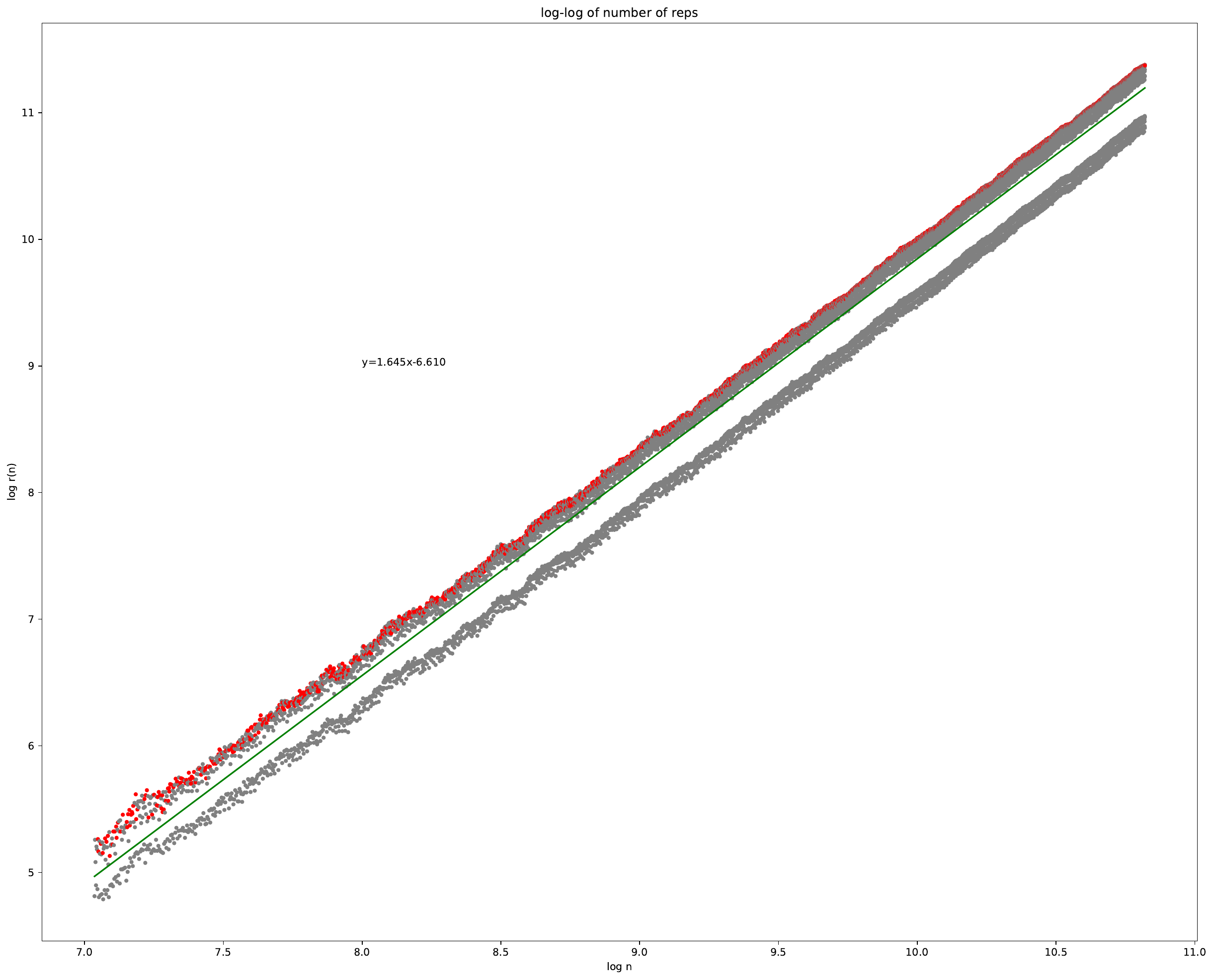}
\end{center}

The figure above shows the log-log-diagram of $r(N)$ in the range between $999$ and $\num[group-separator={,}]{49999}$ and its (green) regression line. The red dots denote the places where $N$ is a prime number. Apparently, asymptotically they establish a red line parallel to the green one, situated on the upper boundary of the diagram. (Compare with the asymptotic formula (1.\,11) in \cite{HL1923} for the number of \textbf{all} representations of an odd number as a sum of three primes.)

In the next section we will see that under the assumption $H(7,6)$ the following holds true:
\[\frac{u(n)}{p(n)}<6\text{ for all }n\geq1\,.\]
\paragraph{Question} Can one find a not too big $K$ such that $H(K,26)$ holds? (NB: Some obvious and easily done calculations let us expect that even $H(387,26)$ holds.)

\paragraph*{Remark} Moreover, every odd integer $N$, $7\leq N\leq\num[group-separator={,}]{99999}$, can be written as the sum of three primes between $\frac N3-N^{\frac12}$ and $\frac N3+N^{\frac12}$. This is related to Oppermann's Conjecture, that for each integer $n>1$ there is at least one prime between $n^2-n$ and $n^2$, and between $n^2$ and $n^2+n$. \vspace{.5cm}\\
In the next section we will see that under the latter assumption conjecture (V1) holds, i.\,e.
\[\max(\frac up)=\frac{163}{47}\,.\]
\section{Effective asymptotic bounds for $\mathbf{\frac up}$}
An asymptotic statement $A(x)$ is called \textit{effective} if one can specify an explicit number $\xi$ such that $A(x)$ holds for all $x>\xi$.

The statements of the following two theorems from O.\,Ramaré and Y.\,Saouter (\cite{rs1975}) as well as from L.\,Schoenfeld (\cite{schoenfeld1976}) are effective:
\begin{theorem} (\cite[Theorem 12]{schoenfeld1976}) 
If $x\geq\num[group-separator={,}]{2010759.9}$, then the interval
\[]x,x(1+\frac1{\num[group-separator={,}]{16597}})[\eqqcolon I(x)\]
contains at least one prime number.
\end{theorem}
In particular: If
\begin{align*}y&\coloneqq\frac{\num[group-separator={,}]{16598}}{\num[group-separator={,}]{16597}}x\\&\geq C\coloneqq\num[group-separator={,}]{2010882}\\&>\frac{\num[group-separator={,}]{16598}}{\num[group-separator={,}]{16597}}\cdot\num[group-separator={,}]{2010759.9}\,,\end{align*}
then the interval
\[]y-\frac1{\num[group-separator={,}]{16598}}y,y[=I(x)\]
contains at least one prime number.

Application of theorem 1 to $y=3p(n)$ with $p(n)\geq\num[group-separator={,}]{670294}$ leads to the following effective asymptotic statement:
\begin{corollary}
If $p(n)\geq\num[group-separator={,}]{670294}$, then
\[\frac{u(n)}{p(n)}>3-\frac3{\num[group-separator={,}]{16598}}\,.\]
\end{corollary}
If one believes in the Riemann hypothesis, then one also has
\begin{theorem} (\cite[Theorem 1]{rs2003})

Under the Riemann hypothesis, for $x\geq2$, the interval
\[]x-\frac85\sqrt x\log x,x]\eqqcolon J(x)\]
contains at least one prime element.
\end{theorem}
From this we easily get, in a manner analogously to corollary 1, the following statement on the magnitude of $\frac{u(n)}{p(n)}$:
\begin{corollary}
Under the assumption of the Riemann hypthesis one has
\[\frac{u(n)}{p(n)}\geq3-\frac18\text{ for all }n>13.\]
\end{corollary}
\begin{proof}
The calculations from appendix \ref{App_A} show that this claim holds true for $13<n<\num[group-separator={,}]{7496}$.

If one applies theorem 2 to $x=3p(n), n>\num[group-separator={,}]{7495}$, one sees like in the proof of proposition 1\,c) that $\frac{u(n)}{p(n)}>3-\frac18$ holds for $n>\num[group-separator={,}]{7495}$ as well:

Let $n>\num[group-separator={,}]{7495}$, hence $p(n)\geq\num[group-separator={,}]{76147}$. By theorem 2, the interval $J(3p(n))$ contains at leeast one prime element. Consequently,
\[u(n)>3p(n)-\frac85\sqrt{3p(n)}\log(3p(n))\,,\]
hence
\[3-\frac{u(n)}{p(n)}<\frac85\sqrt3\frac{\log(3p(n))}{\sqrt{p(n)}}\eqqcolon\delta(n)\,.\]
For $p(n)=\num[group-separator={,}]{76147}$, one has $\delta(n)<0.124<\frac18$. Furthermore, for $n\geq5$, $\delta(n)$ is monotonically decreasing.
\end{proof}
With the help of an effective version of ``Vinogradov’s theorem with almost equal summands'' in the sense of \cite[Theorem 1.1]{mmv2017} (which we could not find in the literature) one could, in a manner similar to the above corollaries, find explicit real numbers $\varepsilon,\nu$ such that
\[\frac{u(n)}{p(n)}<3+\epsilon\text{ for all }n>\nu.\]
As seen in section 5, $H(7,6)$ is equivalent to:

Every odd integer $N\geq7$ can be written as the sum $N=q(1)+q(2)+q(3)$ with prime numbers $q(i)$ such that
\[\tag*{$\left(\overset{\displaystyle \ast}\ast\right)$}\frac N6<q(i)<\frac N2\text{ for }i=1,2,3.\]
\begin{corollary}
Under the assumption $H(7,6)$ one has
\[\frac{u(n)}{p(n)}<6\text{ for all }n>0\,.\]
\end{corollary}
This result is clearly much weaker than (V1).
\begin{proof}
Assume $u(n)>6p(n)$, for some $n>0$. Then one would have $u(n)>6$ and, by $H(7,6)$,
\[u(n)=q(1)+q(2)+q(3)\]
with prime numbers $q(i)$ such that, because of $\left(\overset{\displaystyle \ast}\ast\right)$,
\[p(n)<\frac{u(n)}6<q(i)<\frac{u(n)}2\]
and $u(n)=q(1)+q(2)+q(3)$ would not be an atom of $S(n)$, contradiction.
\end{proof}
\paragraph{Question} Can one specify a concrete $K$ for which $H(K,26)$ holds? A positive answer would imply:
\begin{corollary}\label{CorH26}
For all $n$ with $n\log n>\frac5{17}K$ one has
\[\frac{u(n)}{p(n)}<\frac{163}{47}\,.\]

If $K$ was \textbf{not too big}, then, finally, one could verify (V1) by means of direct calculation.
\end{corollary}
\begin{proof}[Proof of corollary \ref{CorH26}] One has $\frac{163}{47}>\frac{17}5$. Let $n\log n>\frac5{17}K$, hence $p(n)>\frac5{17}K$. Assume $u(n)\geq\frac{163}{47}p(n)>\frac{17}5p(n)>K$.

By $H(K,26)$, $u(n)$ can be written as a sum $u(n)=q(1)+q(2)+q(3)$ with prime numbers $q(i)$ such that
\[|\frac{u(n)}3-q(i)|\leq\frac{u(n)}{26}\text{, hence also}\]
\[q(i)\geq\frac{u(n)}3-\frac{u(n)}{26}=\frac{23}{78}u(n)\geq\frac{23}{78}\cdot\frac{17}5p(n)>p(n)\]
and $u(n)$ would not be atom, contradiction.
\end{proof}
\appendix
\section{Appendix}
\label{App_A}
The program listed below can be run using the \emph{GAP} software system (\cite{gap}) where the package \emph{NumericalSgps} (\cite{NumericalSgps}) has been installed:
\begin{lstlisting}{gapcode}
LoadPackage("NumericalSgps");
primes := Filtered([0..1000000],IsPrime);
n := 0;
while n < 7496 do
    n := n+1; #makes p the n-th prime
    p := primes[n];
    S := NumericalSemigroup(Filtered([p..6*p],IsPrime));
    #by the assertion below f < 4*p, so u<6*p
    #and S is the generated by all primes >=p
    g := GenusOfNumericalSemigroup(S);
    m := Multiplicity(S);
    f := FrobeniusNumber(S);
    A := MinimalGeneratingSystemOfNumericalSemigroup(S); #Atoms
    e := Length(A);
    u := Maximum(A);
    s := 1+f-g;
    Assert(0, n<2 or u>2*p);
    Assert(0, s>n-1);
    Assert(0, e>n);
    Assert(0, (n=8  and 5*n=4*e)    or (5*n<4*e));
    Assert(0, (n=15 and 47*u=163*p) or (n<>15 and 47*u<163*p));
    Assert(0, (n=8  and 19*f=63*p)  or (n<>8  and 19*f<63*p));
    Assert(0, (n=8  and 19*s=24*p)  or (n<>8  and 19*s<24*p));
    Assert(0, (n<=46) or ( ((u - 3*p)^10) <  p^7 ));
    Assert(0, (n=30 and f-3*p=n) or (f-3*p<n));
    Assert(0, (n<=30) or ((f-3*p)^10 < p^7 ));
    Assert(0, (n<=13) or (8*u >= 23*p));
od;
\end{lstlisting}

\end{document}